\documentclass{article}
\usepackage{amsfonts}
\usepackage{amsmath}

\newtheorem{theorem}{Theorem}

\newtheorem{corollary}[theorem]{Corollary}

\newtheorem{lemma}[theorem]{Lemma}

\newenvironment{proof}[1][Proof]{\noindent\textbf{#1.} }{\ \rule{0.5em}{0.5em}}
\input{tcilatex}

\begin{document}

\title{On the discrete boundary value problem for anisotropic equation}
\author{Marek Galewski, Szymon G\l \c{a}b}
\maketitle

\begin{abstract}
Using critical point theory methods we undertake the existence and
multiplicity of solutions for discrete anisotropic two-point boundary value
problems.
\end{abstract}

\section{Introduction}

In this note we consider an anisotropic difference equation 
\begin{equation}
\left\{ 
\begin{array}{l}
\Delta \left( |\Delta u(k-1)|^{p(k-1)-2}\Delta u(k-1)\right) =\lambda
f(k,u(k)), \\ 
u(0)=u(T+1)=0,%
\end{array}%
\right.   \label{zad}
\end{equation}%
where $\lambda >0$ is a numerical parameter, $f:[0,T+1]\times \mathbb{R}%
^{T+2}\rightarrow \mathbb{R}$, $[a,b]$ for $a<b$, $a,b\in 
%TCIMACRO{\U{2124} }%
%BeginExpansion
\mathbb{Z}
%EndExpansion
$ denotes a discrete interval $\{a,a+1,...,b\},$ $\Delta u(k-1)=u\left(
k\right) -u(k-1)$ is the forward difference operator; $p:\left[ 0,T+1\right]
\rightarrow \mathbb{R}_{+}$, $p^{-}=\min_{k\in \left[ 0,T+1\right] }p\left(
k\right) >1$ and $p^{+}=\max_{k\in \left[ 0,T+1\right] }p\left( k\right) >1$%
; $q:\left[ 0,T+1\right] \rightarrow \mathbb{R}_{+}$ stands for the
conjugate exponent. We aim at showing that problem (\ref{zad}) has at least
one, at least two nontrivial solutions. The approach relays on the
application of the direct method of the calculus of variations, the mountain
pass technique and next the version of a finite dimensional three critical
point theorem which we derive for our special case.

Let us mention, far from being exhaustive, the following recent papers on
discrete BVPs investigated via critical point theory, \cite{agrawal}, \cite%
{caiYu}, \cite{Liu}, \cite{sehlik}, \cite{TianZeng}, \cite{teraz}, \cite%
{zhangcheng}, \cite{nonzero}. These papers employ in the discrete setting
the variational techniques already known for continuous problems of course
with necessary modifications. The tools employed cover the Morse theory,
mountain pass methodology, linking arguments. We can mention the following
works concerning problems similar to (\ref{zad}), \cite{KoneOuro}, \cite{MRT}%
, where the approach is variational. This is in strict contrast with the
continuous counterpart of (\ref{zad}) - see for example \cite{hasto} and its
vast references.

\section{Variational framework for (\protect\ref{zad})}

By a solution to (\ref{zad}) we mean such a function $x:[0,T+1]\rightarrow 
\mathbb{R}$\ which satisfies the given equation and the associated boundary
conditions. Solutions will be investigated in a space 
\begin{equation*}
H=\{u:[0,T+1]\rightarrow \mathbb{R}:u(0)=u(T+1)=0\}
\end{equation*}%
considered with a norm 
\begin{equation*}
||u||=\left( \sum_{k=1}^{T+1}|\Delta u(k-1)|^{2}\right) ^{1/2}.
\end{equation*}%
Then $(H,||\cdot ||)$ becomes a Hilbert space. The functional corresponding
to (\ref{zad}) is 
\begin{equation*}
J_{\lambda }(u)=\sum_{k=1}^{T+1}\frac{1}{p(k-1)}|\Delta
u(k-1)|^{p(k-1)}-\lambda \sum_{k=1}^{T}F(k,u(k)),
\end{equation*}%
where $F(k,u(k))=\int_{0}^{u(k)}f(k,t)dt$. With any fixed $\lambda >0$
functional $J_{\lambda }$ is differentiable in the sense of G\^{a}teaux and
its G\^{a}teaux derivative reads%
\begin{equation*}
\left\langle J_{\lambda }^{^{\prime }}(u),v\right\rangle
=\sum_{k=1}^{T+1}|\Delta u(k-1)|^{p(k-1)-2}\Delta u(k-1)\Delta
v(k-1)-\lambda \sum_{k=1}^{T}f(k,u(k))v\left( k\right) .
\end{equation*}%
A critical point to $J_{\lambda }$, i.e. such a point $u\in E$ that 
\begin{equation*}
\left\langle J_{\lambda }^{^{\prime }}(u),v\right\rangle =0\text{ for all }%
v\in E
\end{equation*}%
is a weak solution to (\ref{zad}). Summing by parts we see that any weak
solution to (\ref{zad}) is in fact a strong one. Hence in order to solve (%
\ref{zad}) we need to find critical points to $J_{\lambda }$ and investigate
their multiplicity.

The following auxiliary result was proved in \cite{MRT}.

\begin{lemma}
\label{lem1}

\begin{itemize}
\item[(a)] There exist two positive constants $C_1$ and $C_2$ such that 
\begin{equation*}
\sum_{k=1}^{T+1}|\Delta u(k-1)|^{p(k-1)}\geq C_1||u||^{p^-}-C_2
\end{equation*}
for every $u\in H$ with $||u||>1$.

\item[(b)] For any $m\geq 2$ there exists a positive constant $c_{m}$ such
that 
\begin{equation*}
\sum_{k=1}^{T}|u(k)|^{m}\leq c_{m}\sum_{k=1}^{T+1}|\Delta u(k-1)|^{m}
\end{equation*}%
for every $u\in H$.
\end{itemize}
\end{lemma}

We note that 
\begin{equation}
2^{m}\sum_{k=1}^{T}|u(k)|^{m}\geq \sum_{k=1}^{T+1}|\Delta u(k-1)|^{m}
\label{eq1}
\end{equation}%
for any $m\geq 2$. Since any two norms on a finite-dimensional Banach space
are equivalent, then there is a positive constant $K_{m}$ (depending on $m$)
with 
\begin{equation*}
\left( K_{m}\right) ^{m}||u||^{m}\geq \sum_{k=1}^{T}|u(k)|^{m}.
\end{equation*}%
It can be verified that 
\begin{equation}
(T+1)^{\frac{2-m}{2m}}||u||\leq \left( \sum_{k=1}^{T+1}|\Delta
u(k-1)|^{m}\right) ^{1/m}\leq (T+1)^{\frac{1}{m}}||u||  \label{relation_m}
\end{equation}%
for any $u\in H$ and any $m\geq 2$.

Let us recall some preliminaries from critical point theory. Let \textup{$E$
be a reflexive Banach space. }Let $J\in C^{1}(E,\mathbb{R})$. \textup{For
any sequence $\{u_{n}\}\subset E$, if $\{J(u_{n})\}$ is bounded and $%
J^{^{\prime }}(u_{n})\rightarrow 0$ as $n\rightarrow \infty $ possesses a
convergent subsequence, then we say $J$ satisfies the Palais--Smale
condition - (PS) condition for short.}

\begin{theorem}
\label{mountain}(Mountain pass lemma)\cite{Ma} Let $J$ satisfy the (PS)
condition. Suppose that

\begin{enumerate}
\item $J(0)=0$;

\item there exist $\rho >0$ and $\alpha >0$ such that $J(u)\geq \alpha $ for
all $u\in E$ with $||u|| =\rho $;

\item there exist $u_{1}$ in $E$ with $||u_{1}||\geq \rho $ such that $%
J(u_{1})<\alpha $.
\end{enumerate}

Then $J$ has a critical value $c\geq \alpha $. Moreover, $c$ can be
characterized as 
\begin{equation*}
\underset{g\in \Gamma }{\inf }\underset{u\in g([1,0])}{\max }J(u),
\end{equation*}%
where $\Gamma =\{g\in C([1,0],E):g(0)=0,g(1)=u_{1}\}$.
\end{theorem}

\begin{theorem}
\label{coercivity}\cite{Ma}\label{LematCritPoint}If functional $%
J:E\rightarrow \mathbb{R}$, $J$ is weakly lower semi-continuous and
coercive, i.e. $\lim_{||x||\rightarrow \infty }J(x)=+\infty $, then there
exist $x_{0}$ such that 
\begin{equation*}
\inf_{x\in E}J(x)=J(x_{0})
\end{equation*}%
and $x_{0}$ is also a critical point of $J$, i.e. $J^{^{\prime }}(x_{0})=0$.
Moreover, if $J$ is strictly convex, then a critical point is unique.
\end{theorem}

\section{Existence of solution by a direct method and a mountain pass lemma}

In this section we are concerned with the applications of Theorems \ref%
{coercivity} and \ref{mountain} in order to get the existence results. We
introduce firstly some assumptions.

\textbf{H1} $a:\left[ 1,T\right] \rightarrow \mathbb{R}_{+}$\textit{, }$b:%
\left[ 1,T\right] \rightarrow \mathbb{R}$\textit{\ are such that }%
\begin{equation*}
|f(k,t)|\leq a(k)|t|^{q(k)}+b(k)\text{ \textit{for all} }t\in \mathbb{R}%
\text{ \textit{and all} }k\in \left[ 1,T\right] \text{.}
\end{equation*}

\textbf{H2} $a_{1}:\left[ 1,T\right] \rightarrow \mathbb{R}_{+}$\textit{, }$%
b_{1}:\left[ 1,T\right] \rightarrow\mathbb{R}$\textit{\ are such that }%
\begin{equation*}
|f(k,t)|\geq a_{1}(k)|t|^{q(k)}+b_{1}(k)\text{ \textit{for all }}t\in\mathbb{%
R}\text{ \textit{and all} }k\in \left[ 1,T\right] \text{.}
\end{equation*}

\textbf{H3 }$c:\left[ 1,T\right] \rightarrow\mathbb{R}_{+}$\textit{\ is such
that }%
\begin{equation*}
|f(k,t)|\geq c(k)|t|^{q(k)}\text{ \textit{for all} }t\in \mathbb{R}\text{ 
\textit{and all} }k\in \left[ 1,T\right] \text{.}
\end{equation*}

\textbf{H4 }$c_{1}:\left[ 1,T\right] \rightarrow\mathbb{R}_{+}$\textit{\ is
such that }%
\begin{equation*}
|f(k,t)|\leq c_{1}(k)|t|^{q(k)}\text{ \textit{for all} }t\in \mathbb{R}\text{
\textit{and all} }k\in \left[ 1,T\right] \text{.}
\end{equation*}

\begin{theorem}
\label{th1} Assume that condition\textbf{\ H1} holds. Then

\begin{itemize}
\item[(\protect\ref{th1}.1)] functional $J_{\lambda }$ is coercive provided $%
p^{-}>q^{+}+1$;

\item[(\protect\ref{th1}.2)] if $p^{-}=q^{+}+1$, then there is $\lambda
^{\star }$ such that for any $\lambda \in (0,\lambda ^{\star })$ functional $%
J_{\lambda }$ is coercive.
\end{itemize}
\end{theorem}

\begin{proof}
Note that%
\begin{equation*}
\begin{array}{l}
\left\vert \sum_{k=1}^{T}F(k,u(k))\right\vert \leq \sum_{k=1}^{T}\left\vert
\int_{0}^{u(k)}f(k,t)dt\right\vert \leq
\sum_{k=1}^{T}\int_{0}^{u(k)}|f(k,t)|dt\leq \bigskip \\ 
\sum_{k=1}^{T}\int_{0}^{u(k)}(a(k)|t|^{q(k)}+b(k))dt=\sum_{k=1}^{T}\frac{%
a(k)|u(k)|^{q(k)+1}}{q(k)+1}+\sum_{k=1}^{T}b(k)u(k)\leq \bigskip \\ 
\frac{a^{+}}{q^{-}+1}\sum_{k=1}^{T}|u(k)|^{q^{+}+1}+b^{+}%
\sum_{k=1}^{T}|u(k)|\leq \bigskip \\ 
\frac{a^{+}c_{q^{+}+1}}{q^{-}+1}\sum_{k=1}^{T+1}|\Delta
u(k)|^{q^{+}+1}+b^{+}c_{1}\sum_{k=1}^{T}|\Delta u(k)|,%
\end{array}%
\end{equation*}%
where constants $c_{q^{+}+1}$ and $c_{1}$ in the last inequality follow from
lemma \ref{lem1}(b). Using the above estimation we obtain%
\begin{equation*}
\begin{array}{l}
J_{\lambda }(u)\geq \sum_{k=1}^{T+1}\frac{1}{p(k-1)}|\Delta
u(k-1)|^{p(k-1)}-\lambda \frac{a^{+}c_{q^{+}+1}}{q^{-}}\sum_{k=1}^{T+1}|%
\Delta u(k)|^{q^{+}+1}\bigskip \\ 
-\lambda b^{+}c_{1}\sum_{k=1}^{T}|\Delta u(k)|\bigskip%
\end{array}%
\end{equation*}%
Using Lemma \ref{lem1}(a) we have 
\begin{equation*}
\sum_{k=1}^{T+1}|\Delta u(k-1)|^{p(k-1)}\geq C_{1}||u||^{p^{-}}-C_{2}.
\end{equation*}%
Hence%
\begin{equation*}
\begin{array}{c}
J_{\lambda }(u)\geq \frac{C_{1}}{p^{+}}||u||^{p^{-}}-C_{2}-\lambda \frac{%
a^{+}c_{q^{+}+1}}{q^{-}+1}\sum_{k=1}^{T+1}|\Delta u(k)|^{q^{+}+1}-\lambda
b^{+}c_{1}\sum_{k=1}^{T}|\Delta u(k)|.\bigskip%
\end{array}%
\end{equation*}%
Thus $J_{\lambda }(u)\rightarrow \infty $ as $||u||\rightarrow \infty $ in
case $p^{-}>q^{+}+1$.\bigskip

Let us now assume that $p^{-}=q^{+}+1$. Then $J_{\lambda }(u)\rightarrow
\infty $ as $||u||\rightarrow \infty $ in case%
\begin{equation*}
\frac{C_{1}}{p^{+}}-\lambda \frac{a^{+}c_{q^{+}+1}}{q^{-}+1}>0\text{.}
\end{equation*}%
Thus there is $\lambda ^{\star }$ such that for any $\lambda \in (0,\lambda
^{\star })$ the functional $J(u)$ is coercive. In fact we may take $\lambda
^{\star }=\frac{C_{1}\left( q^{-}+1\right) }{p^{+}a^{+}c_{q^{+}+1}}$.
\end{proof}

Now, we immediately obtain the following existence result.

\begin{theorem}
Assume that condition \textbf{H1} holds and let $f\left( k,0\right) \neq 0$
for at least one $k\in \left[ 1,T\right] $. Then problem (\ref{zad}) has at
least one nontrivial solution for all $\lambda >0$ provided $p^{-}>q^{+}+1$.
When $p^{-}=q^{+}+1$, there is $\lambda ^{\star }$ such that for any $%
\lambda \in (0,\lambda ^{\star })$ problem (\ref{zad}) has at least one
nontrivial solution.
\end{theorem}

\begin{proof}
Indeed, functional $J_{\lambda }$ is continuous differentiable in the sense
of G\^{a}teaux. The assertion follows than by Theorem \ref{th1} and Theorem %
\ref{coercivity}.
\end{proof}

\begin{theorem}
\label{th2} Assume that condition \textbf{H2} holds. Then

\begin{itemize}
\item[(\protect\ref{th2}.1)] functional $J_{\lambda }$ is anti-coercive
provided $p^{+}<q^{-}+1$;

\item[(\protect\ref{th2}.2)] if $p^{+}=q^{-}+1$, then there is $\lambda
^{\star }$ such that for any $\lambda >\lambda ^{\star }$ the functional $%
J_{\lambda }$ is anti-coercive.
\end{itemize}
\end{theorem}

\begin{proof}
As in the proof of Theorem \ref{th1} we see by (\ref{eq1}) that%
\begin{equation*}
\begin{array}{l}
\sum_{k=1}^{T}F(k,u(k))\geq \bigskip \\ 
\frac{a_{1}^{-}}{(q^{+}+1)2^{q^{-}+1}}\sum_{k=1}^{T}|\Delta
u(k-1)|^{q^{-}+1}+\frac{b_{1}^{-}}{2}\sum_{k=1}^{T}|\Delta u(k-1)|.%
\end{array}%
\end{equation*}%
Using the above we have for some constant $\overline{c}$ 
\begin{equation*}
J_{\lambda }(u)\leq \frac{1}{p^{+}}||u|| _{p^{+}}^{p^{+}}-\lambda \frac{%
a_{1}^{-}}{(q^{+}+1)2^{q^{-}+1}}||u|| _{q^{-}+1}^{q^{-}+1}+\overline{c}%
\bigskip ||u||\text{.}
\end{equation*}%
Hence by (\ref{relation_m}) we see that 
\begin{equation}
J_{\lambda }(u)\leq \frac{T+1}{p^{-}}||u||^{p^{+}}-\frac{\lambda
a_{1}^{-}(T+1)^{\frac{2}{2-q^{-}+1}}}{(q^{+}+1)2^{q^{-}+1}}||u||^{q^{-}+1}+%
\overline{c}\bigskip ||u||.  \label{oszac_z_goryMPL}
\end{equation}%
This inequality provides the assertion.
\end{proof}

The existence result immediately follows.

\begin{theorem}
Assume that condition \textbf{H2} holds and let $f\left( k,0\right) \neq 0$
for at least one $k\in \left[ 1,T\right] $. Then problem (\ref{zad}) has at
least one nontrivial solution for all $\lambda >0$ provided $p^{+}<q^{-}+1$.
When $p^{+}=q^{-}+1$, there is $\lambda ^{\star }$ such that for any $%
\lambda >\lambda ^{\star }$ problem (\ref{zad}) has at least one nontrivial
solution.
\end{theorem}

When $f\left( k,0\right) =0$ for all $k\in \left[ 1,T\right] $ Theorem \ref%
{coercivity} yields the existence of at least one solution which in this
case may become trivial. So that we will use some different methodology
pertaining to Theorem \ref{mountain}. This however requires some additional
assumptions.

\begin{corollary}
\label{cor1}

\begin{itemize}
\item[(i)] Assume that conditions \textbf{H1}, \textbf{H3} hold and let $%
p^{-}>q^{+}+1$. Then for any $M<0$ there is positive $\lambda ^{\star }$
such that for any $\lambda >\lambda ^{\star }$ and for all $||u||=1$ we have 
$J_{\lambda }(u)\leq M$.

\item[(ii)] Assume that conditions \textbf{H2}, \textbf{H4} hold and let $%
q^{-}+1>p^{+}$. Then there is $\lambda ^{\star }$, $t>0$ and $M>0$ such that
for any $\lambda \in (0,\lambda ^{\star })$ and for all $||u||=t$ we have $%
J_{\lambda }(u)\geq M.$
\end{itemize}
\end{corollary}

\begin{proof}
We will prove only the first assertion. Let us fix any $M>0$. Note that by
the proof of Theorem \ref{th2}, see relation (\ref{oszac_z_goryMPL}), we
obtain that for all $||u||=1$%
\begin{equation*}
J_{\lambda }\left( u\right) \leq \frac{T+1}{p^{-}}-\frac{\lambda c^{-}(T+1)^{%
\frac{2}{2-q^{-}+1}}}{(q^{+}+1)2^{q^{-}+1}}\leq M
\end{equation*}%
when $\lambda \geq \lambda ^{\ast }=\frac{\left( -M+\frac{T+1}{p^{-}}\right)
(q^{+}+1)2^{q^{-}+1}}{c^{-}(T+1)^{\frac{2}{2-q^{-}+1}}}$.
\end{proof}

\begin{corollary}
\begin{itemize}
\item 

\item[(i)] Assume that conditions \textbf{H1,H3} hold and let $p^{-}>q^{+}+1$
and $f\left( k,0\right) =0$ for $k\in \left[ 1,T\right] $. Then problem (\ref%
{zad}) has at least one nontrivial solution for all $\lambda >\lambda ^{\ast
}=\frac{\left( -M+\frac{T+1}{p^{-}}\right) (q^{+}+1)2^{q^{-}+1}}{c^{-}(T+1)^{%
\frac{2}{2-q^{-}+1}}}$.

\item[(ii)] Assume that conditions \textbf{H2}, \textbf{H4} hold and let $%
q^{-}+1>p^{+}$. Then there is $\lambda ^{\star }$ such that (\ref{zad}) has
at least one nontrivial solution for all $\lambda \in (0,\lambda ^{\star })$.
\end{itemize}
\end{corollary}

\begin{proof}
We will show that all assumptions of Theorem \ref{mountain} are satisfied in
case of assertion (i). We put $E_{\lambda }=-J_{\lambda }$. Now we see that $%
E_{\lambda }$ is anti-coercive. Next we note that $E_{\lambda }\left(
0\right) =0$. Thus assumptions 1. and 3. of Theorem \ref{mountain} are
satisfied. Assumption 2. follows by Lemma \ref{cor1}. Hence the assertion of
the theorem follows.
\end{proof}

\section{General multiplicity result}

In this section we are concerned with the existence of multiple solutions.
We will use the version of the three critical point theorem. In order to do
so we shall somehow generalize the main result, namely Theorem 3 from \cite%
{CIT} which is proved for some special functional containing term $%
\left\Vert \cdot \right\Vert ^{p}$. We note that however, that this term can
be replaced by any nonnegative, coercive function being $0$ at $0$. Our
proof in fact follows the ideas employed in \cite{CIT}. Let us start from
the following abstract result which is used in the proof Theorem 3 from \cite%
{CIT} and which we also use.

\begin{theorem}
\label{abstract} Let $(X,\tau )$ be a Hausdorff space and $\Phi
,J:X\rightarrow \mathbb{R}$ be functionals; moreover, let $M$ be the
(possibly empty) set of all the global minimizers of $J$ and define 
\begin{equation*}
\alpha =\inf_{x\in X}\Phi (x),
\end{equation*}%
\begin{equation*}
\beta =\left\{ 
\begin{array}{ll}
\inf_{x\in M}\Phi (x),\mbox{ if }M\neq \emptyset , &  \\ 
\sup_{x\in X}\Phi (x),\mbox{ if }M=\emptyset . & 
\end{array}%
\right.
\end{equation*}%
Assume that the following conditions are satisfied:

\begin{itemize}
\item[(\protect\ref{abstract}.1)] for every $\sigma >0$ and every $\rho\in%
\mathbb{R}$ the set $\{x\in X:\Phi(x)+\sigma J(x)\leq\rho\}$ is sequentially
compact (if not empty);

\item[(\protect\ref{abstract}.2)] $\alpha<\beta$.
\end{itemize}

Then at least one of the following conditions holds:

\begin{itemize}
\item[(\protect\ref{abstract}.3)] there is a continuous mapping $%
h:(\alpha,\beta)\to X$ with the following property: for every $%
t\in(\alpha,\beta)$, one has 
\begin{equation*}
\Phi(h(t))=t
\end{equation*}
and for every $x\in\Phi^{-1}(t)$ with $x\neq h(t)$, 
\begin{equation*}
J(x)>J(h(t));
\end{equation*}

\item[(\protect\ref{abstract}.4)] there is $\sigma ^{\star }>0$ such that
the functional $\Phi +\sigma ^{\star }J$ admits at least two global
minimizers in $X$.
\end{itemize}
\end{theorem}

\begin{theorem}
\label{abstrac1} Let $H$ be a finite dimensional Banach space. Assume that $%
J,\mu \in C^{1}(H)$, $\mu \left( x\right) \geq 0$ for all $x\in H$, $\mu
(0)=0$ and $\mu $ is coercive. Let $0<r<s$ be constants. Assume that

\begin{itemize}
\item[(\protect\ref{abstrac1}.1)] $\limsup_{\mu(u)\to\infty}\frac{J(u)}{%
\mu(u)}\geq 0$;

\item[(\protect\ref{abstrac1}.2)] $\inf_{u\in H}J(u)<\inf_{\mu(u)\leq s}J(u)$%
;

\item[(\protect\ref{abstrac1}.3)] $J(0)\leq\inf_{r\leq\mu(u)\leq s}J(u)$.
\end{itemize}

Then there is $\lambda^\star>0$ such that the functional $%
u\mapsto\mu(u)+\lambda^\star J(u)$ has at least three critical points in $H$.
\end{theorem}

\begin{proof}
Define a continuous functional $\Phi :H\rightarrow 
%TCIMACRO{\U{211d} }%
%BeginExpansion
\mathbb{R}
%EndExpansion
$ in the following way 
\begin{equation*}
\Phi (u)=\left\{ 
\begin{array}{ll}
\mu (u),\mbox{ if }\mu (u)<r, &  \\ 
r,\mbox{ if }r\leq \mu (u)\leq s, &  \\ 
\mu (u)-s+r,\mbox{ if }\mu (u)>s. & 
\end{array}%
\right.
\end{equation*}%
Let $M$, $\alpha $ and $\beta $ be defined as in Theorem \ref{abstract}. At
first we show that $\beta >r$. We consider two cases: \newline
\textbf{Case 1.} Assume that $M\neq \emptyset $. 
%Since the functional $\mu $
%is coercive and by condition (1), functional $\Phi $ is also coercive. 
The set $M$ of minimizers of continuous functional $J$ is closed, and
therefore there is $\overline{u}\in M$ with $\Phi (\overline{u})=\beta $.
Hence by (\ref{abstrac1}.2) we obtain that $\mu (\overline{u})>s$, and thus 
\begin{equation*}
\beta =\Phi (\overline{u})=\mu (\overline{u})-s+r>r.
\end{equation*}%
\textbf{Case 2.} Assume that $M=\emptyset $. Then $\beta =\infty >r$.

Now we will show that the assumptions of Theorem \ref{abstract} are
fulfilled. Note that for any $\sigma >0$ we have 
\begin{equation*}
\begin{array}{l}
\bigskip \liminf_{\mu (u)\rightarrow \infty }(\Phi (u)+\sigma
J(u))=\liminf_{\mu (u)\rightarrow \infty }(\mu (u)-s+r+\sigma J(u))= \\ 
\bigskip \liminf_{\mu (u)\rightarrow \infty }\mu (u)\left( 1-\frac{s-r}{\mu
(u)}+\sigma \frac{J(u)}{\mu (u)}\right) \geq \\ 
\liminf_{\mu (u)\rightarrow \infty }\mu (u)\cdot \liminf_{\mu (u)\rightarrow
\infty }\left( 1-\frac{s-r}{\mu (u)}+\sigma \frac{J(u)}{\mu (u)}\right) \geq
\liminf_{\mu (u)\rightarrow \infty }\mu (u)=\infty .%
\end{array}%
\end{equation*}%
Hence by the above and by the continuity of $\Phi $ and $J$ we obtain that
the set $\{u\in X:\Phi (u)+\sigma J(u)\leq \rho \}$ is closed and bounded,
and consequently compact.

By the definition of $\Phi$ we have $\alpha=\inf_{u\in H}\Phi(u)=0$. Thus
since $\beta>r>0$ we obtain that $\beta>\alpha$. By Theorem \ref{abstract}
one of the two conditions (\ref{abstract}.3) and (\ref{abstract}.4) holds.
We will show that (\ref{abstract}.3) cannot hold. Suppose to the contrary
that (\ref{abstract}.3) holds true, i.e. there is a mapping $h:(0,\beta)\to
H $ with $\Phi(h(t))=t$ for every $t\in(0,\beta)$. Since $r\in(0,\beta)$,
then 
\begin{equation*}
t<r\iff\Phi(h(t))<r\iff \mu(h(t))<r,
\end{equation*}
\begin{equation*}
t=r\iff\Phi(h(t))=r\iff r\leq\mu(h(t))\leq s,
\end{equation*}
\begin{equation*}
t>r\iff\Phi(h(t))>r\iff \mu(h(t))>s.
\end{equation*}
Therefore 
\begin{equation*}
\limsup_{t\to r^-}\mu(h(t))\leq r=\mu(h(r)),
\end{equation*}
\begin{equation*}
\liminf_{t\to r^+}\mu(h(t))\geq s>\mu(h(r)).
\end{equation*}
But this contradicts the continuity of $h$. Hence the condition (\ref%
{abstract}.4) holds, that is there is $\sigma^\star>0$ such that the
functional $\Phi+\sigma^\star J$ has at least two minimizers, say $u_1$ and $%
u_2$ ($u_1\neq u_2$).

We will show that $\mu (u_{i})<r$ or $\mu (u_{i})>s$ for $i=1,2$. Suppose to
the contrary, that $r\leq \mu (u_{i})\leq s$ for some $i=1,2$. Then 
\begin{equation*}
\Phi (u_{i})+\sigma ^{\star }J(u_{i})=r+\sigma ^{\star }J(u_{i})>\sigma
^{\star }J(u_{i})=\Phi (0)+\sigma ^{\star }J(0),
\end{equation*}%
which contradicts with the fact that $u_{i}$ is a minimizer of $\Phi +\sigma
^{\star }J$. Put 
\begin{equation*}
E(u)=\mu (u)+\sigma ^{\star }J(u).
\end{equation*}%
Note that $u_{1}$ and $u_{2}$ are \emph{local} minimizers of $E$. Moreover $%
E\in C^{1}(H,\mathbb{R})$. We need only to show that $E$ has at least one
critical point $u_{3}\in H\setminus \{u_{1},u_{2}\}$. If both $u_{1}$ and $%
u_{2}$ are strict local minimizers of $E$, then using Mountain Pass
technique (see \cite[Theorem 2]{CIT}) we will find a critical point $%
u_{3}\notin \{u_{1},u_{2}\}$. If one of $u_{1}$ and $u_{2}$ is not strict
local minimizer, then $E$ admits infinitely many local minimizers at the
same level.
\end{proof}

\section{Applications of multiplicity result for anisotropic problems}

In order to apply Theorem \ref{abstrac1} for (\ref{zad}) we introduce the
following notation. Let $\mu (u)=\sum_{k=1}^{T+1}\frac{1}{p(k-1)}|\Delta
u(k-1)|^{p(k-1)}$. Clearly $\mu \in C^{1}(H,\mathbb{R})$, $\mu (0)=0$, $\mu
\geq 0$ and $\mu $ is coercive. For numbers $r,s>0$ we put $r^{\prime }=\inf
\{||u||_{\infty }:\mu (u)\geq r\}$ and $s^{\prime }=\sup \{||u||_{\infty
}:\mu (u)\leq s\}$. Hence 
\begin{equation*}
\mu (u)\geq r\implies ||u||_{\infty }\geq r^{\prime },
\end{equation*}%
\begin{equation*}
\mu (u)\leq s\implies ||u||_{\infty }\leq s^{\prime }.
\end{equation*}%
Let $F(k,t)=\int_{0}^{t}f(k,\tau )d\tau $, $J(u)=-\sum_{k=1}^{T}F(k,u(k))$,
and $E_{\lambda }(u)=\mu (u)+\lambda J(u)$.

\begin{theorem}
\label{th3} Let $0<r<s$. Assume that

\begin{itemize}
\item[(\protect\ref{th3}.1)] $\limsup_{|t|\rightarrow \infty }\frac{F(k,t)}{%
|t|^{p^{-}}}\leq 0$ for any $k\in \lbrack 1,T]$;

\item[(\protect\ref{th3}.2)] $\sum_{k=1}^{T}\sup_{|t|\leq s^{\prime
}}F(k,t)<\sum_{k=1}^{T}\sup_{t\in \mathbb{R}}F(k,t)$;

\item[(\protect\ref{th3}.3)] $\sup_{r^{\prime }\leq t\leq s^{\prime
}}F(k,t)\leq -\sum_{h\neq k}\sup_{|t|\leq s^{\prime }}F(h,t)$ for any $k\in
\lbrack 1,T]$.
\end{itemize}

Then there exists $\lambda ^{\star }$ such that the functional $E_{\lambda
^{\star }}$ has at least three critical points, two of which must be
nontrivial.
\end{theorem}

\begin{proof}
By Lemma \ref{lem1}(a), $\mu (u)\geq c_{1}||u||^{p^{-}}-c_{2}$ for some
positive constants $c_{1},c_{2}$ and $||u||>1$. Since there is a positive $%
c_{3}$ with 
\begin{equation*}
||u||^{p^{-}}\geq c_{3}\sum_{k=1}^{T}|u(k)|^{p^{-}},
\end{equation*}%
then 
\begin{equation*}
\mu (u)\geq c_{1}c_{3}\sum_{k=1}^{T}|u(k)|^{p^{-}}-c_{2}\geq \frac{c_{1}c_{3}%
}{2}\sum_{k=1}^{T}|u(k)|^{p^{-}}
\end{equation*}%
provided $\mu (u)$ is sufficiently large.

Let $\varepsilon >0$. Using (\ref{th3}.1) we will find $K>0$ with 
\begin{equation*}
\frac{F(k,t)}{|t|^{p^{-}}}<\frac{c_{1}c_{3}\varepsilon }{2T}
\end{equation*}%
for any $k\in \lbrack 1,T]$ and $|t|>K$. Let $M=\max \{F(k,t):k\in \lbrack
1,T],|t|\leq K\}$. Then for $\mu (u)>MT/\varepsilon $ with $||u||\geq 1$ we
obtain 
\begin{equation*}
\frac{J(u)}{\mu (u)}=\frac{-\sum_{k=1}^{T}F(k,u(k))}{\mu (u)}\geq -\frac{%
\sum_{k=1}^{T}|F(k,u(k))|}{\mu (u)}\geq
\end{equation*}%
\begin{equation*}
-\frac{\sum_{|u(k)|\leq K}|F(k,u(k))|}{\mu (u)}-\frac{%
\sum_{|u(k)|>K}|F(k,u(k))|}{\frac{c_{1}c_{3}}{2}\sum_{k=1}^{T}|u(k)|^{p^{-}}}%
\geq
\end{equation*}%
\begin{equation*}
-\sum_{|u(k)|\leq K}\frac{M}{MT/\varepsilon }-\sum_{|u(k)|>K}\frac{%
|F(k,u(k))|}{\frac{c_{1}c_{3}}{2}|u(k)|^{p^{-}}}\geq
\end{equation*}%
\begin{equation*}
-\sum_{|u(k)|\leq K}\frac{\varepsilon }{T}-\sum_{|u(k)|>K}\frac{\varepsilon 
}{T}=-\varepsilon .
\end{equation*}

Hence we get condition (\ref{abstrac1}.1).

To show condition (\ref{abstrac1}.2) we consider two cases. \newline
\textbf{Case 1.} Suppose that $\inf_{u\in H}J(u)>-\infty $. We will show
that for any $\sigma >0$ the following equality holds 
\begin{equation}
\inf_{||u||_{\infty }\leq \sigma }J(u)=-\sum_{k=1}^{T}\sup_{|t|\leq \sigma
}F(k,t).  \label{eq2}
\end{equation}%
Note that for any $||u||_{\infty }\leq \sigma $ we have 
\begin{equation*}
J(u)=-\sum_{k=1}^{T}F(k,u(k))\geq -\sum_{k=1}^{T}\sup_{|t|\leq \sigma }F(k,t)
\end{equation*}%
On the other hand for any $\varepsilon >0$ and $k\in \lbrack 1,T]$ there is $%
|t_{k}|\leq \sigma $ with 
\begin{equation*}
F(k,t_{k})>\sup_{|t|\leq \sigma }F(k,t)-\frac{\varepsilon }{T}.
\end{equation*}%
Define $\tilde{u}\in H$ by $\tilde{u}(k)=t_{k}$ for $k\in \lbrack 1,T]$.
Then $||\tilde{u}||_{\infty }\leq \sigma $ and 
\begin{equation*}
J(\tilde{u})=-\sum_{k=1}^{T}F(k,\tilde{u}(k))<-\sum_{k=1}^{T}\sup_{|t|\leq
\sigma }F(k,t)-\varepsilon .
\end{equation*}%
Hence we obtain (\ref{eq2}). Similarly one can show that 
\begin{equation}
\inf_{u\in H}J(u)=-\sum_{k=1}^{T}\sup_{t\in \mathbb{R}}F(k,t).  \label{eq3}
\end{equation}%
Now, using (\ref{th3}.2), (\ref{eq2}) and (\ref{eq3}) we obtain 
\begin{equation*}
\inf_{u\in H}J(u)=-\sum_{k=1}^{T}\sup_{t\in \mathbb{R}}F(k,t)<-%
\sum_{k=1}^{T}\sup_{|t|\leq s^{\prime }}F(k,t)=\inf_{||u||_{\infty }\leq
s^{\prime }}J(u)\leq \inf_{\mu (u)\leq s}J(u).
\end{equation*}%
\textbf{Case 2.} Suppose that $\inf_{u\in H}J(u)=-\infty $. Then (\ref%
{abstrac1}.2) holds since continuous functional $J$ attains its minimum on
compact set $\{u:||u||_{\infty }\leq s^{\prime }\}$.

Now we will show (\ref{abstrac1}.3). Note that $J(0)=0$. For any $u\in H$
with $r\leq \mu (u)\leq s$ we have $r^{\prime }\leq ||u||_{\infty }\leq
s^{\prime }$. Take $k\in \lbrack 1,T]$ with $||u||_{\infty }=u(k)$. By (\ref%
{th3}.3) we obtain 
\begin{equation*}
J(u)=-\sum_{h=1}^{T}F(h,u(h))=-F(k,u(k))-\sum_{h\neq k}F(h,u(h))\geq
\end{equation*}%
\begin{equation*}
-\sup_{r^{\prime }\leq |t|\leq s^{\prime }}F(k,t)-\sum_{h\neq
k}\sup_{|t|\leq s^{\prime }}F(h,t)\geq 0.
\end{equation*}%
Finally, by Theorem \ref{abstrac1} there is $\lambda ^{\star }>0$ such that
the functional $E_{\lambda ^{\star }}$ has at least three critical points in 
$H$.
\end{proof}

\begin{corollary}
\label{cor} Let $0<r^{\prime }<s^{\prime }$. Suppose that there is a
positive constant $c$ such that 
\begin{equation*}
r^{\prime }<\frac{1}{2}\left( \frac{cp^{-}}{T+1}\right) ^{\frac{1}{p^{+}}}%
\mbox{ and
}s^{\prime }>\frac{T+1}{2}(cp^{+})^{\frac{1}{p^{-}}}
\end{equation*}%
and conditions (\ref{th3}.1)--(\ref{th3}.3) are fulfilled. Then there is $%
\lambda ^{\star }>0$ such that $E_{\lambda ^{\star }}$ has at least three
critical points, two of which must be nontrivial.
\end{corollary}

\begin{proof}
Let $r>0$. Then $\mu (u)\geq r$ is equivalent to 
\begin{equation*}
\sum_{k=1}^{T+1}\frac{1}{p(k-1)}|\Delta u(k-1)|^{p(k-1)}\geq r.
\end{equation*}%
Therefore there is $k\in \lbrack 1,T+1]$ with 
\begin{equation*}
\frac{1}{p(k-1)}|\Delta u(k-1)|^{p(k-1)}\geq \frac{r}{T+1}.
\end{equation*}%
Hence 
\begin{equation*}
|\Delta u(k-1)|\geq \left( \frac{rp(k-1)}{T+1}\right) ^{\frac{1}{p(k-1)}}.
\end{equation*}%
Thus there is $k\in \lbrack 1,T]$ with 
\begin{equation*}
|u(k)|\geq \frac{1}{2}\left( \frac{rp(k-1)}{T+1}\right) ^{\frac{1}{p(k-1)}%
}\geq \frac{1}{2}\left( \frac{rp^{-}}{T+1}\right) ^{\frac{1}{p^{+}}}.
\end{equation*}%
Consequently 
\begin{equation}
\inf \{||u||_{\infty }:\mu (u)\geq r\}\geq \frac{1}{2}\left( \frac{rp^{-}}{%
T+1}\right) ^{\frac{1}{p^{+}}}.  \label{eq4}
\end{equation}

Now, let $s>0$. Then $\mu (u)\leq s$ is equivalent to 
\begin{equation*}
\sum_{k=1}^{T+1}\frac{1}{p(k-1)}|\Delta u(k-1)|^{p(k-1)}\leq s.
\end{equation*}%
Hence 
\begin{equation*}
\frac{1}{p(k-1)}|\Delta u(k-1)|^{p(k-1)}\leq s
\end{equation*}%
for every $k\in \lbrack 1,T+1]$. Thus for every $k\in \lbrack 1,T+1]$ we
obtain 
\begin{equation*}
|\Delta u(k-1)|\leq (sp(k-1))^{\frac{1}{p(k-1)}}.
\end{equation*}%
Note that if $|\Delta u(k-1)|\leq a$ for some $a$ and every $k\in \lbrack
1,T+1]$, then $|u(k)|\leq \frac{T+1}{2}a$. Using this observation we obtain
that 
\begin{equation*}
|u(k)|\leq \frac{T+1}{2}(sp(k-1))^{\frac{1}{p(k-1)}}.
\end{equation*}%
Thus 
\begin{equation*}
||u||_{\infty }\leq \frac{T+1}{2}(sp^{+})^{\frac{1}{p^{-}}}.
\end{equation*}%
Therefore 
\begin{equation*}
\sup \{||u||_{\infty }:\mu (u)\leq s\}\leq \frac{T+1}{2}(sp^{+})^{\frac{1}{%
p^{-}}}.
\end{equation*}

Suppose that $c>0$ is such that 
\begin{equation*}
r^{\prime }<\frac{1}{2}\left( \frac{cp^{-}}{T+1}\right) ^{\frac{1}{p^{+}}}%
\mbox{ and
}s^{\prime }>\frac{T+1}{2}(cp^{+})^{\frac{1}{p^{-}}}
\end{equation*}%
Since the following functions $t\mapsto \inf \{||u||_{\infty }:\mu (u)\geq
t\}$ and $t\mapsto \sup \{||u||_{\infty }:\mu (u)\leq t\}$ are continuous
and their ranges are equal to $[0,\infty )$, then there are $r<c<s$ with 
\begin{equation*}
\sup \{||u||_{\infty }:\mu (u)\leq s\}=s^{\prime }\mbox{ and }\inf
\{||u||_{\infty }:\mu (u)\geq r\}=r^{\prime }.
\end{equation*}
\end{proof}

\textbf{Example.} Consider the following equation 
\begin{equation*}
\left\{%
\begin{array}{ll}
\Delta\left(|\Delta u(0)|^{2}\Delta u(0)\right)=2\lambda\left(\frac{u^3(1)}{%
10}-u(1)\right), &  \\ 
\Delta\left(|\Delta u(1)|^{3}\Delta u(1)\right)=4\lambda\left(u(2)-\frac{1}{%
10}\right)^3, &  \\ 
u(0)=u(3)=0. & 
\end{array}
\right.
\end{equation*}
Then $T=2$, $p(0)=4$, $p(1)=5$, $p^-=4$, $p^+=5$, $f(1,t)=2\left(\frac{t^3}{%
10}-t\right)$ and $f(2,t)=4\left(t-\frac{1}{10}\right)^3$. Thus 
\begin{equation*}
F(1,t)=\frac{t^4}{20}-t^2\mbox{ and }F(2,t)=\frac{1}{10^4}- \left(t-\frac{1}{%
10}\right)^4.
\end{equation*}
Taking $c=1$ we obtain 
\begin{equation*}
\frac{1}{2}\left(\frac{cp^-}{T+1}\right)^\frac{1}{p^+}=\frac12\left(\frac43%
\right)^\frac15 \approx 0.5296 \mbox{ and } \frac{T+1}{2}(cp^+)^\frac{1}{p^-}%
=\frac32 5^\frac14\approx 2.2430.
\end{equation*}
Take $r^{\prime }=0.2$ and $s^{\prime }=3$. Note that $\sup_{0.2\leq|t|\leq
3}F(1,t)<0$, $\sup_{|t|\leq 3}F(1,t)=0$ and $\sup_{t\in\mathbb{R}%
}F(1,t)=\infty$. Moreover $\sup_{0.2\leq|t|\leq 3}F(2,t)=0$ and $%
\sup_{|t|\leq 3}F(1,t)=\sup_{t\in\mathbb{R}}F(1,t)=\frac{1}{10^4}$. Hence by
Corollary \ref{cor} the equation has at least three solutions for some $%
\lambda^\star>0$.

\begin{tabular}{ll}
\begin{tabular}{l}
Marek Galewski \\ 
Institute of Mathematics, \\ 
Technical University of Lodz, \\ 
Wolczanska 215, 90-924 Lodz, Poland, \\ 
marek.galewski@p.lodz.pl%
\end{tabular}
& 
\begin{tabular}{l}
Szymon G\l \c{a}b \\ 
Institute of Mathematics, \\ 
Technical University of Lodz, \\ 
Wolczanska 215, 90-924 Lodz, Poland, \\ 
szymon.glab@p.lodz.pl%
\end{tabular}%
\end{tabular}

\end{document}